\documentclass[a4 paper,11 pt,reqno]{amsart}
\usepackage{amsmath,hyperref,graphics}
\usepackage{amsfonts}
\usepackage{amsthm}
\usepackage{mathrsfs}  
\newtheorem{theorem}{Theorem}[section]
\newtheorem{lemma}[theorem]{Lemma}

\numberwithin{equation}{section} \makeatletter
\@namedef{subjclassname@2010}{ 2010 Mathematics Subject
Classification} \makeatother
\title[An inequality for the polar derivative of a polynomial]{On an inequality concerning the polar derivative of a polynomial with restricted zeros}
\author{N. A. Rather and Suhail Gulzar}
 \address{Department of Mathematics \\
    University of Kashmir \\
   Srinagar, Hazratbal 190006
   \\ India}
 \email{dr.narather@gmail.com}
 \email{sgmattoo@gmail.com}
\date{}
\begin{document}
\maketitle
\begin{abstract}
 \indent Let $D_\alpha P(z)=nP(z)+(\alpha-z)P^{\prime}(z)$ denote the polar derivative of a polynomial $P(z)$ of degree $n$ with respect to a point $\alpha\in\mathbb{C}.$ In this paper, we present a correct proof, independent of Laguerre's theorem, of an inequality concerning the polar derivative of a polynomial with restricted zeros recently formulated by K. K. Dewan, Naresh Singh, Abdullah Mir, [Extensions of some polynomial inequalities to the polar derivative, \emph{J. Math. Anal. Appl.,} \textbf{352} (2009) 807-815].
\end{abstract}
\section{Introduction}
If $P(z)$ is a polynomial of degree $n$ having all its zeros in $|z|\leq 1,$ it was shown by Tur\'{a}n \cite{t} that
\begin{equation}\label{2}
\underset{\left|z\right|=1}{\max}\left|P^{\prime}(z)\right|\geq \dfrac{n}{2}\,\underset{\left|z\right|=1}{\max}\left|P(z)\right|.
\end{equation}
The inequality \eqref{2} becomes equality for $P(z)=(z+1)^n.$\\
\indent As an extension of \eqref{2}, Malik \cite{m} proved that if all the zeros of polynomial $P(z)$ of degree $n$ lie in $|z|\leq k,\,k\leq 1,$ then
\begin{equation}\label{3}
\underset{\left|z\right|=1}{\max}\left|P^{\prime}(z)\right|\geq \dfrac{n}{1+k}\, \underset{\left|z\right|=1}{\max}\left|P(z)\right|. 
\end{equation}
\indent By considering the class of polynomials $P(z)=a_nz^n+\sum_{j=\mu}^{n}a_{n-j}z^{n-j},$ $1\leq\mu\leq n,$ of degree $n$ having all their zeros in $|z|\leq k,$ $k\leq 1,$ Aziz and Shah \cite{as97} proved
\begin{equation}\label{4}
\underset{\left|z\right|=1}{\max}\left|P^{\prime}(z)\right|\geq        \frac{n}{1+k^\mu}\left\{\underset{\left|z\right|=1}{\max}\left|P(z)\right|+\dfrac{1}{k^{n-\mu}}\underset{|z|=k}{\min}|P(z)|\right\}.
\end{equation}
\indent Let $D_\alpha P(z)$ denote the polar derivative of a polynomial $P(z)$ of degree $n$ with respect to a point $\alpha\in\mathbb{C},$ then (see \cite{marden})
$$D_\alpha P(z)=nP(z)+(\alpha-z)P^{\prime}(z). $$
The polynomial $D_\alpha P(z)$ is of degree at most $n-1$ and it generalizes the ordinary derivative in the sense that
$$\underset{\alpha\rightarrow\infty}{\lim}\dfrac{D_\alpha P(z)}{\alpha}=P^{\prime}(z) $$
uniformly with respect to $z$ for $|z|\leq R,\,R>0.$\\
\indent Dewan et al. \cite{dnr} (see also \cite{rm}) extended inequality \eqref{4} to polar derivative and proved that if $P(z)=a_nz^n+\sum_{j=\mu}^{n}a_{n-j}z^{n-j},$ $1\leq\mu\leq n,$ is a polynomial of degree $n$ having all its zeros in $|z|\leq k,$ $k\leq 1,$ then for every real and complex number $\alpha$ with $|\alpha|\geq k^\mu,$ 
\begin{equation}\label{6}
\underset{\left|z\right|=1}{\max}\left|D_\alpha P(z)\right|\geq        \frac{n\left(|\alpha|-k^\mu\right)}{1+k^\mu}\underset{\left|z\right|=1}{\max}\left|P(z)\right|+\frac{n\left(|\alpha|+1\right)}{k^{n-\mu}\left(1+k^\mu\right)}\underset{|z|=k}{\min}|P(z)|.
\end{equation}
\indent While seeking the desired refinement of inequality \eqref{6}, recently Dewan et al. \cite{dna} have made an incomplete attempt by claiming to have proved the following result.
\begin{theorem}\label{t1}
If $P(z)=a_nz^n+\sum_{j=\mu}^{n}a_{n-j}z^{n-j},$ $1\leq\mu\leq n,$ is a polynomial of degree $n$ having all its zeros in $|z|\leq k,$ $k\leq 1,$ then for every real or complex number $\alpha$ with $|\alpha|\geq k^\mu,$ 
 \begin{align}\nonumber\label{te1}
 \underset{\left|z\right|=1}{\max}\left|D_\alpha P(z)\right|\geq &       \frac{n\left(|\alpha|-k^\mu\right)}{1+k^\mu}\underset{\left|z\right|=1}{\max}\left|P(z)\right|+\frac{n\left(|\alpha|+1\right)}{k^{n-\mu}\left(1+k^\mu\right)}m\\
 & +n\left(\dfrac{k^\mu-A_\mu}{1+k^\mu}\right)\underset{\left|z\right|=1}{\max}|P(z)|+\dfrac{n(A_\mu-k^\mu)}{k^n(1+k^\mu)}m
 \end{align}
 where $m=\underset{|z|=k}{\min}|P(z)|$ and 
 \begin{equation}\label{amu}
 A_\mu=\frac{n\left(|a_n|-\frac{m}{k^n}\right)k^{2\mu}+\mu|a_{n-\mu}|k^{\mu-1}}{n\left(|a_n|-\frac{m}{k^n}\right)k^{\mu-1}+\mu|a_{n-\mu}|}.
 \end{equation}
\end{theorem} 

 The proof of Theorem \ref{t1} given by Dewan et al. \cite{dna} is not correct. The reason being that the authors in \cite{dna}, deduce by using Lemma 7\cite{dna} between the lines 8 to 10, on page 814, that if $F(z)=P(z)-m\lambda z^n/k^n$ has all its zeros in $|z|<k, k\leq 1,$ then for every $\alpha\in\mathbb{C}$ with $|\alpha|\geq k^\mu,$ $1\leq \mu\leq n$ all the zeros of the polar derivative $D_\alpha F(z)$ also lie in $|z|<k$ which need not be true in general for $1\leq \mu\leq n.$ Here, Lemma 7 \cite{dna} is applicable only when $\mu=1.$ For $1<\mu\leq n$ the assertion is not true, since for $k<1,$ $k^\mu< k.$\\
  \indent The immediate counterexample $P(z)=4z^2-1,$ $\mu=2$ having all its zeros in $|z|<k=3/5<1$ demonstrates, by taking $\alpha=2/5> k^\mu$ that the zero of $D_\alpha P(z)=\frac{16z}{5}-2$ lie in $|z|>k=3/5.$\\
  \indent  The main aim of this paper is to present a correct proof, independent of Laguerre's theorem, of Theorem \ref{t1}.
\section{Lemmas}
For the proof of above theorem, we need the following lemmas. 
\begin{lemma}\label{l1}
If $P(z)=a_nz^n+\sum_{j=\mu}^{n}a_{n-j}z^{n-j},$ $1\leq\mu\leq n,$ is a polynomial of degree $n$ having all its zeros in $|z|\leq k,$ $k\leq 1$ and $Q(z)=z^n\overline{P(1/\overline{z})},$ then for $|z|=1$
\begin{equation}
|Q^{\prime}(z)|\leq S_\mu|P^{\prime}(z)|\,\left(\leq k^\mu|P^{\prime}(z)| \right)
\end{equation}
where 
\begin{equation}\label{smu}
 S_\mu=\frac{n|a_n|k^{2\mu}+\mu|a_{n-\mu}|k^{\mu-1}}{n|a_n|k^{\mu-1}+\mu|a_{n-\mu}|}
\end{equation}
and
\begin{equation*}
\frac{\mu}{n}\left|\dfrac{a_{n-\mu}}{a_n}\right|\leq k^\mu.
\end{equation*}
\end{lemma}
The above lemma is due to Aziz and Rather \cite{ar04}. 
\begin{lemma}\label{l2}
Let $P(z)=a_nz^n+\sum_{j=\mu}^{n}a_{n-j}z^{n-j},$  $1\leq\mu\leq n,$ be a polynomial of degree $n$ having all its zeros in $|z|\leq k,$ $k\leq 1,$ then for every real or complex number $\alpha$ with $|\alpha|\geq S_\mu,$ 
\begin{align}\label{l1e}
|D_\alpha P(z)|\geq n\left(\frac{|\alpha|-S_\mu}{1+k^\mu}\right)|P(z)|\qquad for\quad|z|=1.
\end{align}
\end{lemma}
\begin{proof}
If $Q(z)=z^n\overline{P(1/\overline{z})},$ then $P(z)=z^n\overline{Q(1/\overline{z})}$ and it can be easily verified that for $|z|=1$
\begin{align*}
|Q^{\prime}(z)|&=|nP(z)-zP^{\prime}(z)|\\&\geq |nP(z)|-|zP^{\prime}(z)|.
\end{align*}
Equivalently, we have
\begin{equation}\label{le1}
|Q^{\prime}(z)|+|P^{\prime}(z)|\geq n|P(z)| \,\,\,\,\,\textnormal{for} \,\,\,\,\,\,\,|z|=1.
\end{equation}
This implies with the help of Lemma \ref{l1} and inequality \eqref{le1}, that
\begin{align*}
\left(1+k^\mu\right)|P^{\prime}(z)|&=|P^{\prime}(z)|+k^\mu|P^{\prime}(z)|\\&\geq|P^{\prime}(z)|+|Q^{\prime}(z)|\\&\geq n|P(z)|,
\end{align*}
which implies 
\begin{equation}\label{le2}
\left|P^{\prime}(z)\right|\geq\dfrac{n}{1+k^\mu}\left|P(z)\right|\,\,\,\,\textnormal{for}\,\,\,\,\,\,|z|=1.
\end{equation}
Now, for every real or complex number $\alpha$ with $|\alpha|\geq S_\mu,$
\begin{align}\nonumber\label{z1}
|D_\alpha P(z)|&=|nP(z)+(\alpha-z)P^{\prime}(z)|\\\nonumber
&\geq|\alpha||P^{\prime}(z)|-|nP(z)-zP^{\prime}(z)|\\&
=|\alpha||P^{\prime}(z)|-|Q^{\prime}(z)|\qquad\textnormal{for}\quad|z|=1.
\end{align}
Combing inequalities \eqref{le2}, \eqref{z1} and Lemma \ref{l1}, we get for $|z|=1,$
\begin{align*}
|D_\alpha P(z)|\geq&\left(|\alpha|-S_\mu\right)|P^{\prime}(z)|\\
\geq& n\left(\dfrac{|\alpha|-S_\mu}{1+k^\mu}\right)\left|P(z)\right|
\end{align*}
This proves Lemma \ref{l2}.
\end{proof}
\begin{lemma}\label{Mm'}
If $P(z)=\sum_{j=1}^{n}a_jz^j$ is a polynomial of degree $n$ having all its zeros in $|z|\leq k,$ $k\leq 1$ and $m={\min}_{|z|=k}|P(z)|,$ then 
\begin{equation}\label{Mm}
\underset{|z|=1}{\max}|P(z)|\geq \dfrac{m}{k^n}
\end{equation}
and in particular,
\begin{equation}\label{Mm''}
|a_n|>\frac{m}{k^n}.
\end{equation}
\end{lemma} 
\begin{proof}
Since the polynomial $P(z)$ has all its zeros in $|z|\leq k,$ $k\leq 1,$ the polynomial $Q(z)=z^n\overline{P(1/\overline{z})}$ has no zero in $|z|<1/k,$ $1/k\geq 1.$ We assume that $Q(z)$ has no zero on $|z|=1/k,$ for otherwise the result holds trivially. Since $Q(z)$ has no zeros in $|z|\leq 1/k,$ by Minimum Modulus Principle
\begin{equation*}
|Q(z)|\geq \underset{|z|=1/k}{\min}|Q(z)|\,\,\,\,\textnormal{for} \,\,\,\,\ |z|\leq 1/k\,\,\,\textnormal{where}\,\,\,\,\, 1/k\geq 1,
\end{equation*}
which in particular gives,
\begin{equation*}
|Q(z)|\geq \dfrac{1}{k^n}\underset{|z|=k}{\min}|P(z)|\,\,\,\,\textnormal{for} \,\,\,\,\ |z|\leq 1\,\,\,\,\ \textnormal{and}\,\,\,\, |a_n|=|Q(0)|>\dfrac{1}{k^n}\underset{|z|=k}{\min}|P(z)|.
\end{equation*}
That is,
\begin{equation*}
\underset{|z|=1}{\max}|P(z)|=\underset{|z|=1}{\max}|Q(z)|\geq \dfrac{m}{k^n}\,\,\,\,\textnormal{and}\,\,\,\,\,|a_n|>\dfrac{m}{k^n}.
\end{equation*}
This completes the proof of Lemma \ref{Mm'}..
\end{proof}
\begin{lemma}\label{l5}
The function
\begin{equation}
S_\mu(x)=\frac{nxk^{2\mu}+\mu|a_{n-\mu}|k^{\mu-1}}{nxk^{\mu-1}+\mu|a_{n-\mu}|}.
\end{equation}
where $k\leq 1$ and $\mu\geq 1,$ is a non-increasing function of $x.$
\end{lemma}
\begin{proof}
The proof follows by considering the first derivative test for $S_\mu(x).$
\end{proof}
  \begin{lemma}\label{l3}
If $P(z)=a_nz^n+\sum_{\nu=\mu}^{n}a_{n-\nu}z^{n-\nu},$ $1\leq\mu\leq n,$ is a polynomial of degree $n,$ having all its zeros in $|z|\leq k,$ $k\leq 1,$ then for every $\alpha\in\mathbb{C}$ with $|\alpha|\geq k^{\mu},$  
\begin{align}
 |D_\alpha P(z)|\geq \dfrac{|\alpha|mn}{k^n}\qquad\textnormal{for}\qquad|z|=1
\end{align}
where $m=\min_{|z|=k}|P(z)|$.
 \end{lemma}
\begin{proof}
By hypothesis all the zeros of polynomial $P(z)=a_nz^n+\sum_{\nu=\mu}^{n}a_{\nu-\mu}z^{\nu-\mu},$ $1\leq\mu\leq n,$ lie in $|z|\leq k,$ $k\leq 1.$ If $P(z)$ has a zero on $|z|=k$ then $m=0$ and the result holds trivially. Hence, we suppose that all the zeros of $P(z)$ lie in $|z|< k,$ $k\leq 1,$ so that $m> 0$ and $m\leq |P(z)|$ for $|z|=k,$ mow if $\lambda$ is any complex number such that $|\lambda|<1,$ then
\begin{equation*}
\left|\dfrac{m\lambda z^n}{k^n}\right|<|P(z)|\qquad \textnormal{for}\qquad |z|=k.
\end{equation*}  
Since all the zeros of $P(z)$ lie in $|z|< k,$ it follows by Rouche's theorem that all the zeros of $$F(z)=P(z)-\dfrac{m\lambda z^n}{k^n}$$ also lie in $|z|<k, k\leq 1.$ Applying Lemma \ref{l1} to the polynomial $F(z),$ we get 
\begin{equation*}
 k^\mu|F^{\prime}(z)|\geq |G^{\prime}(z)| \qquad \textnormal{for}\qquad |z|=1.
\end{equation*}
where $G(z)=z^n\overline{F(1/\overline{z})}=z^n\overline{P(1/\overline{z})}+\frac{m\overline{\lambda}}{k^n}=Q(z)++\frac{m\overline{\lambda}}{k^n}.$ 
Equivalently for $|z|=1,$ we have
\begin{align}\label{lp1}
 k^\mu\left|P^{\prime}(z)-\dfrac{mn\lambda z^{n-1}}{k^n}\right|\geq |Q^{\prime}(z)|.
\end{align}
Moreover, by Gauss-Lucas theorem that all the zeros of polynomial $F^{\prime}=P^{\prime}(z)-mn\lambda z^{n-1}/k^n$ lie in $|z|<k,$ $k\leq 1$ for every $\lambda$ with $|\lambda|<1.$ This implies 
\begin{align}\label{lp2}
|P^{\prime}(z)|\geq mn|z|^{n-1}/k^n\,\,\,\,\textnormal{for}\,\,\,\,|z|\geq k.
\end{align}
Now, choosing the argument of $\lambda$ in the left hand side of \eqref{lp1} such that 
\begin{align*}
k^\mu\left|P^{\prime}(z)-\dfrac{mn\lambda z^{n-1}}{k^n}\right|=k^\mu\left\{|P^{\prime}(z)|-\dfrac{|\lambda|mn}{k^n}     \right\}\,\,\,\,\textnormal{for}\,\,\,\,\,|z|=1,
\end{align*}
which is possible by \eqref{lp2}, we get
\begin{align*}
k^\mu|P^{\prime}(z)|-\dfrac{|\lambda|mn}{k^{n-\mu}}\geq |Q^{\prime}(z)|\,\,\,\,\,\,\,\textnormal{for}\,\,\,\,|z|=1.
\end{align*}
Letting $|\lambda|\rightarrow 1,$ we obtain 
\begin{align}\label{lp3}
k^\mu|P^{\prime}(z)|-\dfrac{mn}{k^{n-\mu}}\geq |Q^{\prime}(z)|\,\,\,\,\,\,\,\textnormal{for}\,\,\,\,|z|=1.
\end{align}
Again, for $|z|=1$ and $\alpha\in\mathbb{C},$ we have
 \begin{align*}
 |D_\alpha P(z)|&=|nP(z)+(\alpha-z)P^{\prime}(z)|\\\nonumber&\geq |\alpha||P^{\prime}(z)|-|nP(z)-zP^{\prime}(z)|\\&=|\alpha||P^{\prime}(z)|-|Q^{\prime}(z)|
 \end{align*}
 Combining this with inequality \eqref{lp3}, we obtain for $|z|=1$ and $\alpha\in\mathbb{C}$ with $|\alpha|\geq k^\mu,$ 
 \begin{align}\label{r1}
 |D_\alpha P(z)|\geq (|\alpha|-k^\mu)|P^{\prime}(z)|+k^\mu\dfrac{mn}{k^{n}}
 \end{align}
 Inequality \eqref{r1} in conjunction with \eqref{lp2} gives for $|z|=1$ and $|\alpha|\geq k^\mu,$
 \begin{align}\label{r2}\nonumber
 |D_\alpha P(z)|&\geq (|\alpha|-k^\mu)\dfrac{mn}{k^n}+k^\mu\dfrac{mn}{k^{n}}\\&=\dfrac{|\alpha|mn}{k^n}.
 \end{align}
 This completes the proof of Lemma \ref{l3}.

\end{proof} 
The next lemma is due to Gardner, Govil and Weems \cite{ggw}.
\begin{lemma}\label{lggw}
If $P(z)=a_0+\sum_{j=\mu}^{n}a_{j}z^{j},$ $1\leq \mu\leq n,$ is a polynomial of degree $n$ having no zero in $|z|\leq k,$ $k\geq 1,$ then 
\begin{align}
 s_0\geq k^\mu 
\end{align}
where $$ s_0=k^{\mu+1}\left\{\frac{\left(\frac{\mu}{n}\right)\dfrac{|a_\mu|}{|a_0|-m}k^{\mu-1}+1}{\left(\frac{\mu}{n}\right)\frac{|a_\mu|}{|a_0|-m}k^{\mu+1}+1}\right\} .$$
\end{lemma}
The following Lemma can be easily verified by applying Lemma \ref{lggw} to $Q(z)=z^n\overline{P(1/\overline{z})}$ where $P(z)=a_nz^n+\sum_{j=\mu}^{n}a_{n-j}z^{n-j},$ $1\leq \mu\leq n.$
\begin{lemma}\label{l4}
If $P(z)=a_nz^n+\sum_{j=\mu}^{n}a_{n-j}z^{n-j},$ $1\leq \mu\leq n,$ is a polynomial of degree $n$ having all its zeros in $|z|\leq k,$ $k\leq 1,$ then \\
\begin{equation}\label{lm7}
 A_\mu\leq k^\mu
\end{equation}
where $A_\mu$ is defined in \eqref{amu}.
\end{lemma}
 \section{Proof of theorem}
 \begin{proof}[\textnormal{\textbf{Proof of theorem \ref{t1}}}]
Proceeding similarly as in the proof of lemma \ref{l3}, the polynomial  $ P(z)-\frac{m\lambda z^n}{k^n} $ has all its zeros in $|z|\leq k,\,k\leq 1,$ for every real or complex $\lambda$ with $|\lambda|<1.$ Applying Lemma \ref{l2} to the polynomial $ P(z)-\frac{m\lambda z^n}{k^n}, $ we obtain for $|\alpha|\geq S^{\prime}_\mu,$
\begin{equation*}
\Big|D_\alpha\Big\{P(z)-\frac{m\lambda z^n}{k^n}\Big\}\Big|\geq n\left(\frac{|\alpha|-S^{\prime}_\mu}{1+k^\mu}\right)\left|P(z)-\frac{m\lambda z^n}{k^n}\right|\quad\textnormal{for}\quad|z|=1
\end{equation*}
where
\begin{equation}\label{p2}
  S^{\prime}_\mu=\frac{n\left|a_n-\frac{m\lambda}{k^n}\right|k^{2\mu}+\mu|a_{n-\mu}|k^{\mu-1}}{n\left|a_n-\frac{m\lambda}{k^n}\right|k^{\mu-1}+\mu|a_{n-\mu}|}.
  \end{equation}
That is,
\begin{equation}\label{p1}
\Big|D_\alpha P(z)-\frac{\lambda mn \alpha z^{n-1}}{k^n}\Big|\geq n\left(\frac{|\alpha|-S^{\prime}_\mu}{1+k^\mu}\right)\left|P(z)-\frac{m\lambda z^n}{k^n}\right|\quad\textnormal{for}\quad|z|=1
\end{equation}
  Since for every $\lambda$ with $|\lambda|<1$, we have by inequality \eqref{Mm''} of Lemma \ref{Mm},
   \begin{equation}\label{p3}
   \left|a_n-\dfrac{m\lambda}{k^n}\right|\geq|a_n|-\dfrac{m|\lambda|}{k^n}\geq|a_n|-\dfrac{m}{k^n}
   \end{equation}
   therefore, it follows by Lemma \ref{l5} that for every $\lambda$ with $|\lambda|<1,$ 
   \begin{align}\label{p4}\nonumber
   S^{\prime}_\mu=&\frac{n\left|a_n-\frac{m\lambda}{k^n}\right|k^{2\mu}+\mu|a_{n-\mu}|k^{\mu-1}}{n\left|a_n-\frac{m\lambda}{k^n}\right|k^{\mu-1}+\mu|a_{n-\mu}|}\\\leq&\frac{n\left(|a_n|-\frac{m}{k^n}\right)k^{2\mu}+\mu|a_{n-\mu}|k^{\mu-1}}{n\left(|a_n|-\frac{m}{k^n}\right)k^{\mu-1}+\mu|a_{n-\mu}|}=A_\mu 
     \end{align} 
     Using \eqref{p4} in \eqref{p1}, we obtain for $|\alpha|\geq A_\mu(\geq S^{\prime}_{\mu}),$
     \begin{equation}\label{p}
     \Big|D_\alpha P(z)-\frac{\lambda mn \alpha z^{n-1}}{k^n}\Big|\geq n\left(\frac{|\alpha|-A_\mu}{1+k^\mu}\right)\left|P(z)-\frac{m\lambda z^n}{k^n}\right|\quad\textnormal{for}\quad|z|=1.
     \end{equation}
    Now, choosing the argument of $\lambda$ in the right hand side of inequality \eqref{p} such that
  \begin{align}\label{p5}
  \Big|D_\alpha P(z)-\frac{\lambda mn \alpha z^{n-1}}{k^n}\Big|=|D_\alpha P(z)|-\frac{|\lambda| |\alpha|mn z^{n-1}}{k^n}
  \end{align} 
  which is possible by Lemma \ref{l3}, we obtain by using Lemma \ref{l4} for $|\alpha|\geq k^\mu(\geq A_\mu)$ and $|z|=1,$
  \begin{align*}
|D_\alpha P(z)|-\frac{|\lambda| |\alpha|mn |z|^{n-1}}{k^n}\geq n\left(\frac{|\alpha|-A_\mu}{1+k^\mu}\right)\left\{|P(z)|-\frac{|\lambda|m |z|^n}{k^n}\right\},
  \end{align*} 
equivalently,
  \begin{align}\label{p6}
|D_\alpha P(z)|\geq n\left(\frac{|\alpha|-A_\mu}{1+k^\mu}\right)|P(z)|+|\lambda|\dfrac{mn}{k^n}\left(\dfrac{|\alpha|k^\mu+A_\mu}{1+k^\mu}\right).
  \end{align}
  Letting $|\lambda|\rightarrow 1,$ in \eqref{p6}, we get
  \begin{align*}
  |D_\alpha P(z)|\geq n\left(\frac{|\alpha|-A_\mu}{1+k^\mu}\right)|P(z)|+\dfrac{mn}{k^n}\left(\dfrac{|\alpha|k^\mu+A_\mu}{1+k^\mu}\right)\quad\textnormal{for}\quad|z|=1,
  \end{align*}
  which implies for every $\alpha\in\mathbb{C}$ with $|\alpha|\geq k^\mu,$ $1\leq \mu\leq n,$
  \begin{align*}
  \underset{\left|z\right|=1}{\max}\left|D_\alpha P(z)\right|\geq &       \frac{n\left(|\alpha|-k^\mu\right)}{1+k^\mu}\underset{\left|z\right|=1}{\max}\left|P(z)\right|-\frac{n\left(|\alpha|+1\right)}{k^{n-\mu}\left(1+k^\mu\right)}m\\
   & +n\left(\dfrac{k^\mu-A_\mu}{1+k^\mu}\right)\underset{\left|z\right|=1}{\max}|P(z)|+\dfrac{n(A_\mu-k^\mu)}{k^n(1+k^\mu)}m.
  \end{align*}
  This completes the proof of Theorem \ref{t1}.
 \end{proof}


\begin{thebibliography}{99}
\bibitem{as97} A. Aziz and W. M. Shah, An integral mean estimate for polynomial, \emph{Indian J. Pure Appl. Math.,} \textbf{28} (1997) 1413-1419.
\bibitem{ar04} A. Aziz and N. A. Rather, Some Zygmund type $L^q$ inequalities for polynomials, \emph{J. Math. Anal. Appl.,} \textbf{289} (2004) 14-29.
\bibitem{dna} K. K. Dewan, Naresh Singh, Abdullah Mir, Extensions of some polynomial inequalities to the polar derivative, \emph{J. Math. Anal. Appl.,} \textbf{352} (2009) 807-815.
\bibitem{dnr} K. K. Dewan, Naresh Singh, Roshan Lal, Inequalities for the polar derivatives of a polynomials, \emph{Int. J. Pure Appl. Math.,} \textbf{33} (2006) 109-117.
\bibitem{ggw} R.B. Gardner, N.K. Govil, A. Weems, Some results concerning rate of growth of polynomials, \emph{ East J. Approx.,} \textbf{10} (2004) 301–312.
\bibitem{m} M. A. Malik, On the derivative of a polynomial, \emph{J. London Math. Soc.,} \textbf{2} (1969) 57–60.
\bibitem{marden} M. Marden, \emph{Geometry of Polynomials}, Math. Surveys No. 3, Amer. Math. Soc. Providence R. I. 1966.
\bibitem{rm} N. A. Rather and M. I. Mir, Some refinements of inequalities for the polar derivative of
polynomials with restricted zeros, \emph{Int. J. Pure and Appl. Math.}, \textbf{41} (2007), 1065-1074.
\bibitem{t} P. Tur\'{a}n, \"{U}ber die Ableitung von Polynomen, \emph{Compositio Mathematica} \textbf{7} (1939), 89-95 (German).
\end{thebibliography}
\end{document}